\documentclass[12pt]{amsart}
\usepackage{amsmath}
\usepackage{mathtools}
\usepackage{amssymb}
\usepackage{amsthm}
\usepackage{graphicx}
\usepackage{enumerate}
\usepackage[mathscr]{eucal}
\usepackage[pagewise]{lineno}
\usepackage{tikz}
\usetikzlibrary{decorations.text,calc,arrows.meta}
\theoremstyle{plain}

\usepackage[english]{babel}
\DeclareMathOperator{\conv}{conv}
\DeclareMathOperator{\Span}{Span}
\DeclareMathOperator{\sgn}{sgn}

\DeclareMathOperator{\Ext}{Ext}

\DeclareMathOperator{\diam}{diam}

\newtheorem{theorem}{Theorem}
\newtheorem{cor}[theorem]{Corollary}

\newtheorem{prop}[theorem]{Proposition}

\theoremstyle{definition}
\newtheorem{definition}[theorem]{Definition}

\usepackage{hyperref}
\hypersetup{
    colorlinks=true,
    linkcolor=blue,
    filecolor=magenta,      
    urlcolor=cyan,
}
\usepackage[pagewise]{lineno}
\usepackage{hyperref}

\begin{document}

\title[Geometry of $\ell_p$-direct sums]{Geometry of $\ell_p$-direct sums of normed linear spaces}

\author[Bose]{Babhrubahan Bose}
                \newcommand{\acr}{\newline\indent}
                

\address[Bose]{Department of Mathematics\\ Indian Institute of Science\\ Bengaluru 560012\\ Karnataka \\INDIA\\ }
\email{babhrubahanb@iisc.ac.in}

 \thanks{The research of Babhrubahan Bose is funded by PMRF research fellowship under the supervision of Professor Apoorva Khare and Professor Gadadhar Misra.}

\begin{abstract}
We consider $\ell_p$-direct sums ($1\leq p<\infty$) and $c_0$-direct sums of countably many normed spaces and find the duals of these spaces. We characterize the support functionals of arbitrary elements in these spaces to characterize smoothness and approximate smoothness, both locally and globally. These results let us obtain examples of spaces that are not approximately smooth but where every non-zero element is approximately smooth. We also characterize Birkhoff-James orthogonality and its pointwise symmetry in these spaces.
\end{abstract}

\maketitle   
   
\section*{Introduction}
The aim of the present article is to study the geometry of the normed linear spaces constructed by taking countably infinite $\ell_p$ direct sums (for $1\leq p<\infty$) of normed linear spaces. We also consider $c_0$ analogues of the direct sums and find the duals of these spaces. We further characterize the support functionals of a non-zero element in these spaces. Consequently, we characterize smoothness and approximate smoothness in them and answer a question about the approximate smoothness of a space raised by Chmieli\'nski, Khurana, and Sain in \cite{approx}. We finish by characterizing Birkhoff-James orthogonality and its pointwise symmetry in these spaces. A similar analysis was done for $\ell_p$ direct sums of a pair of normed linear spaces in \cite{approx} by Chmlie\'nski, Khurana, and Sain, where the support functionals and approximate smoothness in these finite direct sums were studied. In this article, we consider $\ell_p$-direct sums of countably many normed linear spaces and study support functionals, approximate smoothness, Birkhoff-James orthogonality, and its pointwise symmetry in such spaces.\par

Let us establish the relevant notations and terminologies to be used throughout the article. Throughout, $\mathbb{K}=\mathbb{R}$ or $\mathbb{C}$. Given a normed linear space $\mathbb{X}$ over $\mathbb{K}$, let $B_\mathbb{X}$ denote the closed unit ball of the space. For any $C\subseteq\mathbb{X}$ convex, we denote the collection of all extreme points of $C$ by $\Ext(C)$. Let $\mathbb{X}^*$ stand for the continuous dual of the space $\mathbb{X}$ and define the support functional of a non-zero element $x\in\mathbb{X}$ to be any $f\in\mathbb{X}^*$ such that
\begin{align*}
    \|f\|=1,~ f(x)=\|x\|.
\end{align*}
Let $J(x)$ denote the collection of support functionals of a non-zero $x$. Clearly, $J(x)$ is convex and weak$^*$ compact. The diameter of $J(x)$ for a non-zero $x$ is denoted by $D(x)$. We also define $\mathcal{D}(\mathbb{X})$ to be the supremum of $D(x)$ over all non-zero $x$, i.e.
\[D(x):=\diam(J(x)),~~\mathcal{D}(\mathbb{X}):=\sup\{D(x):x\in\mathbb{X}\setminus\{0\}\}.\]
A non-zero element $x\in\mathbb{X}$ is said to be \textit{smooth} if it has a unique support functional. Hence, a non-zero $x$ is smooth if and only if $D(x)=0$. A normed space $\mathbb{X}$ is \textit{smooth} if every non-zero element of the space is smooth, i.e., $\mathcal{D}(\mathbb{X})=0$. Let us recall the definition of approximate smoothness as proposed by Chmieli\'nski, Khurana and Sain in \cite{approx}.
\begin{definition}
Let $\epsilon\in[0,2)$. A non-zero $x\in\mathbb{X}$ is called \textit{$\epsilon$-smooth} if $D(x)\leq\epsilon$. A non-zero $x\in\mathbb{X}$ is called \textit{approximately smooth} if $x$ is $\epsilon$-smooth for some $\epsilon\in[0,2)$. The space $\mathbb{X}$ is said to be \textit{approximately smooth ($\epsilon$-smooth)} if $\mathcal{D}(\mathbb{X})\leq\epsilon$ for some $0\leq\epsilon<2$.
\end{definition}
Given two elements $x,y\in\mathbb{X}$, $x$ is defined to be \textit{Birkhoff-James orthogonal} to $y$ \cite{B}, denoted by $x\perp_By$ if 
\[\|x+\lambda y\|\geq\|x\|,~\textit{for every scalar}~\lambda.\]
James proved in \cite{james} that $x\perp_By$ if and only if $x=0$ or $f(y)=0$ for some support functional $f$ of $x$. In the same article, he proved that a non-zero point $x\in\mathbb{X}$ is smooth if and only if Birkhoff-James orthogonality is right additive at $x$, i.e., for any $y,z\in\mathbb{X}$,
\begin{align*}
    x\perp_By,~x\perp_Bz~\Rightarrow~x\perp_B(y+z).
\end{align*}
James proved in \cite{james2} that in a normed linear space of dimension 3 or more, Birkhoff-James orthogonality is symmetric if and only if the space is an inner product space. However, the importance of studying the point-wise symmetry of Birkhoff-James orthogonality in describing the geometry of normed linear spaces has been illustrated in \cite[Theorem 2.11]{CSS}, \cite[Corollary 2.3.4]{Sain}. Let us recall the following definition in this context from \cite{Sain2}, which will play an important part in our present study.
\begin{definition}
An element $x$ of a normed linear space $\mathbb{X}$ is said to be \textit{left-symmetric} (\textit{resp. right-symmetric}) if 
\begin{align*}
    x\perp_By\;\Rightarrow\; y\perp_Bx~~(\textit{resp.~}y\perp_Bx\;\Rightarrow\;x\perp_By),
\end{align*}
for every $y\in \mathbb{X}$.
\end{definition}
Note that we refer to the left-symmetric and right-symmetric points of a given normed linear space by the term {\textit{point-wise symmetry of Birkhoff-James orthogonality}. Birkhoff-James orthogonality and its pointwise symmetry have been the focus of tremendous research aimed at understanding the geometry of a normed space. We refer the readers to \cite{me}, \cite{usseq}, \cite{CSS},\cite{dkp}, \cite{1}, \cite{3}, \cite{4}, \cite{5}, \cite{KP}, \cite{Sain2}, \cite{Sain}, \cite{8}, \cite{10}, \cite{SRBB}, \cite{turnsek}, \cite{12}  for some of the prominent works in this direction.\par
A \textit{semi-inner product} on a $\mathbb{K}$ vector space $\mathbb{V}$ is defined to be a map $[\cdot,\cdot]:\mathbb{V}\times\mathbb{V}\to\mathbb{K}$ such that for $x,y,z\in\mathbb{V}$ and $\lambda\in\mathbb{K}$,
\begin{enumerate}
    \item $[x,x]\geq0$ with equality if and only if $x=0$.
    \item $[x,y]+\lambda[x,z]=[x,y+\lambda z]$.
    \item $[\lambda x,y]=\overline{\lambda} [x,y]$
    \item $|[x,y]|^2\leq[x,x][y,y]$.
\end{enumerate}
A \textit{semi-inner product} on a Banach space $\mathbb{X}$ is a map $[\cdot,\cdot]:\mathbb{X}\times\mathbb{X}\to\mathbb{K}$ satisfying the above four properties along with $[x,x]=\|x\|^2$ for every $x\in\mathbb{X}$. Construction of a semi-inner product on $\mathbb{X}$ requires a map $\Psi:\mathbb{K}\mathbb{P}\mathbb{X}\to S_{\mathbb{X}^*}$ such that $\Psi([x])$ is the support functional of some $x_0\in [x]\cap S_\mathbb{X}$, where $\mathbb{K}\mathbb{P}\mathbb{X}$ denotes the $\mathbb{K}$-projective space of $\mathbb{X}$ \footnote{Recall, this is the set of equivalence classes in $\mathbb{X}$ under the relation of multiplication by nonzero scalars in $\mathbb{K}$ -- i.e., the set of lines in $\mathbb{X}$.} and $[x]$ denotes the equivalence class of $x\in\mathbb{X}\setminus \{0\}$ in $\mathbb{K}\mathbb{P}\mathbb{X}$. Note that the element $x_0$ is uniquely determined by the map $\Psi$ and the element $[x]\in\mathbb{K}\mathbb{P}\mathbb{X}\setminus \{0\}$. The semi-inner product can then be constructed as
\begin{align*}
    [x,y]:=
\begin{cases}
\overline{\lambda}\left(\Psi([x])\right)(y);~x=\lambda x_0,~(\Psi([x]))(x_0)=1,~~x\neq0~x,y\in\mathbb{X}\\
0,~x=0,~y\in\mathbb{X}.
\end{cases}
\end{align*}
Note that a non-zero point $x\in\mathbb{X}$ is smooth if and only if 
\begin{align*}
    [x,y]_1=[x,y]_2,~\textit{for every}~y\in\mathbb{X},
\end{align*}
for every pair $[\cdot,\cdot]_1$, $[\cdot,\cdot]_2$ of semi-inner products on $\mathbb{X}$. Also, for $x,y\in\mathbb{X}$, $x\perp_By$ if and only if $[y,x]=0$ for some semi-inner product $[\cdot,\cdot]$ on $\mathbb{X}$.\par
We now define a notion of pointwise symmetry of semi-inner products. 
\begin{definition}\label{sip-commute}
    Let $1<p<\infty$. A point $x\in\mathbb{X}$ is said to be \textit{$p$-left (resp.$p$-right) s.i.p. commuting} with $y\in\mathbb{X}\setminus\{0\}$ if given any semi-inner product $[\cdot,\cdot]$, there exists a semi-inner product $[\cdot,\cdot]'$ such that
    \[[y,x]'=\left|\frac{[x,y]}{\|x\|\|y\|}\right|^{p-2}\overline{[x,y]}~~\left(\textit{resp.}~[y,x]=\left|\frac{[x,y]'}{\|x\|\|y\|}\right|^{p-2}\overline{[x,y]'}\right).\]
    A point $x\in\mathbb{X}$ is said to be \textit{$p$-left (resp. $p$-right) s.i.p. symmetric} if $x$ is $p$-left (resp. $p$-right) s.i.p. commuting with every $y\in\mathbb{X}$. Also, if $x\in\mathbb{X}$ is said to be \textit{left (resp. right) s.i.p. symmetric} if $x$ is $2$-left (resp. $2$-right) s.i.p. symmetric.
\end{definition}We can immediately note the following proposition relating s.i.p. symmetry with pointwise symmetry of Birkhoff-James orthogonality.
\begin{prop}
    If $x\in\mathbb{X}$ is $p$-left (resp. $p$-right) s.i.p. symmetric, then $x$ is a left-symmetric (resp. right-symmetric) point.
\end{prop}
\begin{proof}
    Follows immediately from the observation that $x\perp_By$ if and only if $[x,y]=0$ for some semi-inner product $[\cdot,\cdot]$.
\end{proof}
\par
\subsection*{Organization of the article}
In the first section, we define the $\ell_p$-direct sums ($1\leq p<\infty$) and $c_0$-direct sums of normed spaces and characterize the duals of these spaces. In the second section, we characterize the support functionals of these elements and obtain values of diameters of any point in the space characterizing approximate smoothness completely. We also use these results to answer a question raised by Chmli\'enski, Khurana, and Sain in \cite{approx}.
In the final section, we characterize Birkhoff-James orthogonality and its pointwise symmetry in these spaces.

\section{$\ell_p$-direct sums}
We now define the $\ell_p$ direct sums of a sequence of normed linear spaces. Throughout, $\mathbb{X}_n$ would denote a sequence of non-trivial (of dimension more than 0) normed linear spaces. We define the following direct sums:
\begin{definition}
    Let $1\leq p<\infty$. Then the \textit{$\ell_p$-direct sum} of $\mathbb{X}_n$ is defined as:
    \begin{align*}
        \bigoplus\limits_{p}\mathbb{X}_n:=\left\{\{x_n\}_{n\in\mathbb{N}}:x_n\in\mathbb{X}_n,~\sum\limits_{n=1}^\infty\|x_n\|^p<\infty\right\}.
    \end{align*}
    Also define:
    \begin{align*}
        \bigoplus\limits_\infty\mathbb{X}_n&:=\left\{\{x_n\}_{n\in\mathbb{N}}:x_n\in\mathbb{X}_n,~\sup\limits_{n=1}^\infty\|x_n\|<\infty\right\},\\
        \bigoplus\limits_0\mathbb{X}_n&:=\left\{\{x_n\}_{n\in\mathbb{N}}:x_n\in\mathbb{X}_n,~\lim\limits_{n\to\infty}\|x_n\|=0\right\}.
    \end{align*}
\end{definition}

We define the norms for these spaces as follows:
\begin{definition}
    Let $1\leq p<\infty$. Then for $\{x_n\}_{n\in\mathbb{N}}$ $\in$ $\bigoplus\limits_{p}\mathbb{X}_n$, define:
    \begin{align*}
        \left\|\{x_n\}_{n\in\mathbb{N}}\right\|_p:=\left(\sum\limits_{n=1}^\infty\|x_n\|^p\right)^\frac{1}{p}.
    \end{align*}
    Similarly, for $\{x_n\}_{n\in\mathbb{N}}$ $\in$ $\bigoplus\limits_{\infty}\mathbb{X}_n$,
    define
    \begin{align*}
        \left\|\{x_n\}_{n\in\mathbb{N}}\right\|_\infty:=\sup\limits_{n=1}^\infty\|x_n\|.
    \end{align*}
    We also define for $\{x_n\}_{n\in\mathbb{N}}$ $\in$ $\bigoplus\limits_{0}\mathbb{X}_n$,
    \begin{align*}
        \left\|\{x_n\}_{n\in\mathbb{N}}\right\|_0:=\max\limits_{n=1}^\infty\|x_n\|.
    \end{align*}
\end{definition}
\begin{prop}
    For $p\in[1,\infty]\cup\{0\}$, $\left(\bigoplus\limits_{p}\mathbb{X}_n,\|\cdot\|_p\right)$ is a normed linear space. Also, it is a Banach space if and only if $\mathbb{X}_n$ is a Banach space for every $n\in\mathbb{N}$.
\end{prop}
\begin{proof}
The first statement follows from the sub-additivity of the norms and Minkowski's inequality. Also, if any $\{x_k^n\}_{k\in\mathbb{N}}\subset\mathbb{X}_n$ is a Cauchy sequence which is not convergent, then defining $\{z^k\}_{k\in\mathbb{N}}\subset\bigoplus\limits_{p}\mathbb{X}_n$ as:
    \begin{align*}
        z^k:=\{z^k_j\}_{j\in\mathbb{N}},~z^k_j:=\begin{cases}
            x_k^n,~j=n,\\
            0,~\text{otherwise},
        \end{cases}
    \end{align*}
    gives a Cauchy sequence in $\bigoplus\limits_{p}\mathbb{X}_n$ which is not convergent. \par
    Conversely, assume all the $\mathbb{X}_n$ to be Banach spaces. Let $x^k=\{x_n^k\}_{n\in\mathbb{N}}\in\bigoplus\limits_p\mathbb{X}_n$ be such that $\{x^k\}_{k\in\mathbb{N}}$ is a Cauchy sequence. Then clearly as $\|x^k\|_p\geq\|x^k_n\|$ for every $k,n\in\mathbb{N}$, $\{x^k_n\}_{n\in\mathbb{N}}$ is a Cauchy sequence in $\mathbb{X}_n$. Let $y_n:=\lim\limits_{k\to\infty}x_n^k$.\par
    Let $p\in[1,\infty)$ and consider $\epsilon>0$. Fix $N\in\mathbb{N}$. Then for $k,j>K$ for $K$ sufficiently large,
    \begin{align*}
        \sum\limits_{n=1}^N\|x_n^k-x_n^j\|^p\leq\sum\limits_{n=1}^\infty\|x_n^k-x_n^j\|^p=\|x^k-x^j\|_p^p<\epsilon^p.
    \end{align*}
    Taking $j\to\infty$, we get
    \begin{align}\label{estimate1}
        \sum\limits_{n=1}^N\|x_n^k-y_n\|^p<\epsilon^p.
    \end{align}
    Hence we have 
    \begin{align*}
        \left(\sum\limits_{n=1}^N\|y_n\|^p\right)^\frac{1}{p}\leq \left(\sum\limits_{n=1}^N\|x_n^k-y_n\|^p\right)^\frac{1}{p}+\left(\sum\limits_{n=1}^N\|x_n^k\|^p\right)^\frac{1}{p}<\epsilon+\|x^k\|_p.
    \end{align*}
    Taking $N\to\infty$ now yields $y=\{y_n\}_{n\in\mathbb{N}}\in\bigoplus\limits_p\mathbb{X}_n$. Now taking $N\to\infty$ in \eqref{estimate1}, we get that for $k>K$,
    \begin{align*}
        \|x^k-y\|_p<\epsilon.
    \end{align*}
    For $p=\infty,0$, the proof proceeds similarly.
\end{proof}

We now characterize the dual of $\bigoplus\limits_p\mathbb{X}_n$ for $p=[1,\infty)\cup\{0\}$. 
\begin{theorem}
    Let $p=[1,\infty)$. Let $q\in(1,\infty]$ such that $\frac{1}{p}+\frac{1}{q}=1$. For $p=0$, set $q=1$ by definition. Then the dual of $\bigoplus\limits_p\mathbb{X}_n$ is isometrically isomorphic to $\bigoplus\limits_q\mathbb{X}_n^*$ with $f=\{f_n\}_{n\in\mathbb{N}}\in\bigoplus\limits_q\mathbb{X}_n^*$ acting on $\bigoplus\limits_p\mathbb{X}_n$ as
    \begin{align*}
        f(x):=\sum\limits_{n=1}^\infty f_n(x_n),~~x=\{x_n\}_{n\in\mathbb{N}}\in\bigoplus\limits_p\mathbb{X}_n.
    \end{align*}
\end{theorem}
\begin{proof}
    First let $1<p<\infty$. Clearly for $f=\{f_n\}_{n\in\mathbb{N}}\in\bigoplus_q\mathbb{X}_n^*$, we have by Minkowski's inequality,
    \begin{align*}
        \left|\sum\limits_{n=1}^\infty f_n(x_n)\right|\leq\sum\limits_{n=1}^\infty\|f_n\|\|x_n\|\leq\|f\|_q\|x\|_p,
    \end{align*}
    for any $x=\{x_n\}_{n\in\mathbb{N}}\in\bigoplus\limits_p\mathbb{X}_n$. Further for $\epsilon>0$, find $x_n\in S_{\mathbb{X}_n}$ such that
    \begin{align*}
        f_n(x_n)>\|f_n\|-\frac{\epsilon \|f\|_q^{q-1}}{2^n\|f_n\|^{q-1}},
    \end{align*}
    whenever $f_n\neq0$ and take $x_n=0$ otherwise.
    Then $y=\{y_n\}_{n\in\mathbb{N}}$ given by $y_n=\frac{\|f_n\|^{q-1}}{\|f\|_q^{q-1}} x_n$ is an element of the unit sphere of $\bigoplus\limits_p\mathbb{X}_n$. Further,
    \begin{align*}
        f(y)=\sum\limits_{n=1}^\infty \frac{\|f_n\|^{q-1}}{\|f\|_q^{q-1}}f_n(x_n)>\sum\limits_{n=1}^\infty\frac{\|f_n\|^{q-1}}{\|f\|_q^{q-1}}\left(\|f_n\|-\frac{\epsilon \|f\|_q^{q-1}}{2^n\|f_n\|^{q-1}}\right)=\|f\|_q-\epsilon.
    \end{align*}
    Hence $f$ is a continuous functional on $\bigoplus\limits_p\mathbb{X}_n$.\par
    Now let $\psi$ be a continuous functional on $\bigoplus\limits_p\mathbb{X}_n$. Then for any $n\in\mathbb{N}$, 
    \[x\mapsto\psi(xe_n),~x\in\mathbb{X}_n,\]
    is a bounded linear functional on $\mathbb{X}_n$, where $xe_n$ stands for the sequence having the $n$-th coordinate $x$ and the rest zero. Let us denote this functional by $\psi_n\in\mathbb{X}_n^*$. Again consider $x_n\in S_{\mathbb{X}_n}$ such that 
    \begin{align*}
        \psi_n(x_n)>\|\psi_n\|-\frac{\epsilon}{2^\frac{n}{p}}.
    \end{align*}
    Then for any $\{c_n\}_{n\in\mathbb{N}}\in\ell_p$, we have 
    \begin{align*}
        \sum\limits_{n=1}^\infty\|\psi_n\||c_n|<\psi\left(\sum\limits_{n=1}^\infty| c_n|x_ne_n\right)+\epsilon\|\{c_n\}_{n\in\mathbb{N}}\|_p<\infty.
    \end{align*}
    Hence $\{\|\psi_n\|\}_{n\in\mathbb{N}}\in\ell_q$. The proofs for the $p=1,0$ cases follow similarly.
\end{proof}

\section{geometry of $\ell_p$ direct sums}
Again, let $\mathbb{X}_n$ be a given sequence of non-trivial (of dimension more than 0) normed spaces. Since our aim is to characterize the smoothness and approximate smoothness, we begin by characterizing the support functionals of a non-zero element of $\bigoplus\limits_p\mathbb{X}_n$. We begin with the $p\in(1,\infty)$ case first and deal with the other cases later.
\begin{theorem}\label{support}
    Let $p\in(1,\infty)$ and let $q$ be its conjugate. Then for any non-zero $x=\{x_n\}_{n\in\mathbb{N}}\in\bigoplus\limits_p\mathbb{X}_n$, and $f=\{f_n\}_{n\in\mathbb{N}}\in\bigoplus\limits_q\mathbb{X}_n^*$, $f\in J(x)$ if and only if 
    \begin{align}\label{supporteq}
        \frac{\|x\|_p^{p-1}}{\|x_n\|^{p-1}}f_n\in J\left(x_n\right),
    \end{align}
    if $x_n\neq0$ and $f_n=0$ otherwise.
\end{theorem}
\begin{proof}
    The sufficiency follows from elementary computations. For the necessity, observe that
    \begin{align*}
        \|x\|_p=f(x)=\sum\limits_{n=1}^\infty f_nx_n&\leq\sum\limits_{n=1}^\infty\|f_n\|\|x_n\|\\
        &\leq\|\left(\sum\limits_{n=1}^\infty \|f_n\|^q\right)^\frac{1}{q}\left(\sum\limits_{n=1}^\infty x_n\|^p\right)^\frac{1}{p}=\|f\|_q\|x\|_p.
    \end{align*}
    Hence, equality must hold in the two above inequalities. From the first inequality, we get
    \[f_n(x_n)=\|f_n\|\|x_n\|,\]
    for every $n\in\mathbb{N}$. Hence $\frac{f_n}{\|f_n\|}\in J(x_n)$ if $x_n\neq0$. Also, from the second inequality, by the condition of equality in Holder's inequality, we get:
    \begin{align*}
        \frac{\|f_n\|^q}{\|x_n\|^p}=\frac{\|f\|_q^q}{\|x\|_p^p}=\frac{1}{\|x\|_p^p}~\Rightarrow~\|f_n\|=\frac{\|x_n\|^{p-1}}{\|x\|^{p-1}},
    \end{align*}
    for every $n\in\mathbb{N}$. Combining the two results yields \eqref{supporteq}.
\end{proof}
Recall that for any non-zero $x$ in a normed space, $D(x):=\diam(J(x))$. We now find $D(x)$ for any $x\in\bigoplus\limits_p\mathbb{X}_n$.
\begin{theorem}\label{size}
    Let $p\in(1,\infty)$. For $x=\{x_n\}_{n\in\mathbb{N}}\in\bigoplus\limits_p\mathbb{X}_n$, 
    \begin{align}\label{sizeeq}
        D(x)=\left(\sum\limits_{n=1}^\infty\frac{\|x_n\|^p}{\|x\|_p^p}(D(x_n))^q\right)^\frac{1}{q}.
    \end{align}
\end{theorem}
\begin{proof}
    Let $f^1=\{f^1_n\}_{n\in\mathbb{N}},f^2=\{f^2_n\}_{n\in\mathbb{N}}\in J(x)$. Then $\frac{\|x\|_p^{p-1}}{\|x_n\|^{p-1}}f^i_n\in J(x_n)$ whenever $x_n\neq 0$ for $i=1,2$. Hence
    \[\|f^1_n-f^2_n\|\leq \frac{\|x_n\|^{p-1}}{\|x\|_p^{p-1}}D(x_n).\]
    Hence
    \begin{align*}
        \|f^1-f^2\|_q^q=\sum\limits_{n=1}^\infty\|f_n^1-f_n^2\|^q\leq\sum\limits_{n=1}^\infty\frac{\|x_n\|^p}{\|x\|_p^p}(D(x_n))^q.
    \end{align*}
    Taking supremum over $f^1,f^2\in J(x)$ now yields:
    \begin{align*}
        D(x)\leq\left(\sum\limits_{n=1}^\infty\frac{\|x_n\|^p}{\|x\|_p^p}(D(x_n))^q\right)^\frac{1}{q}.
    \end{align*}
    Again, fix $\epsilon>0$. Find $f^1_n,f^2_n\in J(x_n)$ such that
    \[\|f^1_n-f^2_n\|^q>(D(x_n))^q-\epsilon.\]
    Set $f^i=\left\{\frac{\|x_n\|^{p-1}}{\|x\|_p^{p-1}}f^i_n\right\}$ for $i=1,2$. Hence we have $f^1,f^2\in J(x)$. However, 
    \begin{align*}
        \|f^1-f^2\|^q=\sum\limits_{n=1}^\infty\frac{\|x_n\|^p}{\|x\|_p^p}\|f^1_n-f^2_n\|^q\leq\left(\sum\limits_{n=1}^\infty\frac{\|x_n\|^p}{\|x\|_p^p}(D(x_n))^q\right)-\epsilon.
    \end{align*}
    Since $\epsilon>0$ is arbitrary, \eqref{sizeeq} follows.
\end{proof}
\begin{cor}
    A non-zero element $x=\{x_n\}_{n\in\mathbb{N}}\in\bigoplus\limits_p\mathbb{X}_n$ is smooth if and only if $x_n\in\mathbb{X}_n$ is smooth whenever $x_n\neq0$.
\end{cor}
Recall that $\mathcal{D}(\mathbb{X})$ is defined to be $\sup\{D(x):x\in\mathbb{X}\setminus\{0\}\}$. We can, therefore, conclude:
\begin{cor}
    Given normed linear spaces $\mathbb{X}_n$, 
    \begin{align*}
        \mathcal{D}\left(\bigoplus\limits_p\mathbb{X}_n\right)=\sup\limits_{n=1}^\infty\mathcal{D}(\mathbb{X}_n).
    \end{align*}
\end{cor}
Note that this answers the question raised by Chmlie\'nski, Khurana, and Sain in \cite{approx}:\\
\textit{``Is there a Banach space $\mathbb{X}$ such that $D(x)<2$ for every non-zero element $x$ but $\mathcal{D}(\mathbb{X})=2?$"}\\
It was shown in \cite{approx} that no finite-dimensional Banach space has this property. Here, we show the existence of an infinite dimensional Banach space having this property:
\begin{prop}
    Let $\mathbb{X}_n$ be a sequence of Banach spaces such that $\mathcal{D}(\mathbb{X}_n)=2-\frac{1}{n}$. (For example, we can choose $\mathbb{X}_n$ to be suitable two-dimensional polygonal spaces as was shown in \cite{approx}.) Then for any $p\in(1,\infty)$, for every $x\in\bigoplus\limits_p\mathbb{X}_n$, $D(x)<2$ but $\mathcal{D}\left(\bigoplus\limits_p\mathbb{X}_n\right)=2$.
\end{prop}
We finish this section by characterizing the support functionals of elements of $\bigoplus\limits_1\mathbb{X}_n$ and $\bigoplus\limits_0\mathbb{X}_n$.
\begin{theorem}\label{support10}
    For $x=\{x_n\}_{n\in\mathbb{N}}\in\bigoplus\limits_1\mathbb{X}_n$, $f=\{f_n\}_{n\in\mathbb{N}}\in\bigoplus\limits_\infty\mathbb{X}_n^*$ is a support functional of $x$ if and only if $f_n\in J(x_n)$ if $x_n\neq0$ and $\|f_n\|\leq1$ if $x_n=0$.\par
    For $x=\{x_n\}_{n\in\mathbb{N}}\in\bigoplus\limits_0\mathbb{X}_n$, $f=\{f_n\}_{n\in\mathbb{N}}\in\bigoplus\limits_1\mathbb{X}_n^*$ is a support functional of $x$ if and only if $f_n=\lambda_nJ(x_n)$ if $\|x_n\|=\|x\|_0$ and $f_n=0$ otherwise such that $\lambda_n>0$, $\sum\limits_{\|x_n\|=\|x\|_0}\lambda_n=1$.
\end{theorem}
\begin{proof}
    The proof follows from direct computations.
\end{proof}
We can now easily compute the value of $D(x)$ for any non-zero  $x\in\bigoplus\limits_1\mathbb{X}_n$. 
\begin{theorem}
    Let $x=\{x_n\}_{n\in\mathbb{N}}\in\bigoplus\limits_1\mathbb{X}_n$ be non-zero. Then
    \begin{align*}
        D(x)=
        \begin{cases}
            \sup\limits_{n=1}^\infty D(x_n),~~\text{if}~x_n\neq0~\text{for every}~n\in\mathbb{N},\\
            2,~~\text{otherwise}.
        \end{cases}
    \end{align*}
    Also for a non-zero $x=\{x_n\}_{n\in\mathbb{N}}\in\bigoplus\limits_0\mathbb{X}_n$,
    \begin{align*}
        D(x)=
        \begin{cases}
            D(x_n)~~\text{if}~\|x_k\|=\|x\|_0~\text{if and only if}~k=n,\\
            2,~\text{otherwise}.
        \end{cases}
    \end{align*}
\end{theorem}
The characterization of smoothness therefore follows:
\begin{cor}
    A point $x=\{x_n\}_{n\in\mathbb{N}}\in\bigoplus\limits_1\mathbb{X}_n$ is smooth if and only if $x_n\in\mathbb{X}_n$ is smooth for every $n\in\mathbb{N}$.\par
    A point $x=\{x_n\}_{n\in\mathbb{N}}\in\bigoplus\limits_0\mathbb{X}_n$ is smooth if and only if there exists a natural number $n_0$ such that $\|x_n\|<\|x\|_0$ for every $n\neq n_0$ and $x_{n_0}\in\mathbb{X}_{n_0}$ is smooth.
\end{cor}
Also, $\bigoplus\limits_p\mathbb{X}_n$ is not approximately smooth for $p=1,0$.
\begin{cor}
    For any sequence of normed linear spaces $\mathbb{X}_n$, 
    \begin{align*}
        \mathcal{D}\left(\bigoplus\limits_1\mathbb{X}_n\right)=\mathcal{D}\left(\bigoplus\limits_1\mathbb{X}_n\right)=2.
    \end{align*}
\end{cor}

\section{Birkhoff-James orthogonality in $\ell_p$-direct sums}
In this section we characterize Birkhoff-James orthogonality and its pointwise symmetry in $x\in\bigoplus\limits_p\mathbb{X}_n$. As usual, let $\mathbb{X}_n$ be a given sequence of non-trivial (of dimension more than 0) normed linear spaces.
We begin the section by characterizing the extreme points of $J(x)$ for any $x\in\bigoplus\limits_p\mathbb{X}_n$.
\begin{prop}
    Let $x=\{x_n\}_{n\in\mathbb{N}}\in\bigoplus\limits_p\mathbb{X}_n$ and let $f=\{f_n\}_{n\in\mathbb{N}}\in\bigoplus\limits_q\mathbb{X}_n^*$, where $p=[1,\infty)\cup\{0\}$ and $q$ is the conjugate index of $p$. Then
    \begin{enumerate}
        \item If $p\in(1,\infty)$ then $f\in\Ext(J(x))$ if and only if 
        \[\frac{\|x\|_p^{p-1}}{\|x_n\|^{p-1}}f_n\in \Ext(J(x_n)),~~\text{for every}~~x_n\neq0.\]
        \item If $p=1$, then $f\in\Ext(J(x))$ if and only if
        \[f_n\in \Ext(J(x_n))~~\text{if}~~x_n\neq0~~\text{and}~~f_n\in\Ext\left(B_{\mathbb{X}_n}\right)~~\text{otherwise.}\]
        \item If $p=0$, then $f\in\Ext(J(x))$ if and only if there exists $n_0\in\mathbb{N}$ such that:
        \[\|x_{n_0}\|=\|x\|_0,~~f_{n_0}\in\Ext\left(J\left(x_{n_0}\right)\right),~~f_n=0~~\text{for}~~n\neq n_0.\]
    \end{enumerate}
\end{prop}
\begin{proof}
    The results follow from Theorems \ref{support} and \ref{support10}.
\end{proof}
We can now characterize Birkhoff-James orthogonality in these spaces.
\begin{theorem}\label{orthogonality}
    Let $x=\{x_n\}_{n\in\mathbb{N}},~y=\{y_n\}_{n\in\mathbb{N}}\in\bigoplus\limits_p\mathbb{X}_n$. Then
    \begin{enumerate}
        \item If $p\in(1,\infty)$, then $x\perp_B y$ if and only if
        \[0\in\overline{\conv}\left\{\sum\limits_{n=1}^\infty\|x_n\|^{p-1}f_n(y_n):f_n\in\Ext(J(x_n))\right\}.\]
        \item If $p=1$, then $x\perp_B y$ if and only if
        \[\left|\sum\limits_{x_n\neq0}f_n(y_n)\right|\leq\sum\limits_{x_n=0}\|y_n\|,\]
        for some $f_n\in J(x_n)$.
        \item If $p=0$, then $x\perp_By$ if and only if
        \[0\in\overline{\conv}\{f_n(y_n):f_n\in\Ext(J(x_n)),~\|x_n\|=\|x\|_0\}.\]
    \end{enumerate}
\end{theorem}
\begin{proof}
    Since $J(x)$ is weak$^*$ compact and convex, $\{f(y):f\in J(x)\}$ is also a compact and convex subset of the ground field, and therefore, it must be the closed convex hull of its extreme points. However, any extreme point of this set must be an image of an extreme point of $J(x)$, finishing our proof for $p\in(1,\infty)\cup\{0\}$.\par
    For $p=1$, if $x\perp_By$ and $f=\{f_n\}_{n\in\mathbb{N}}\in J(x)$ is such that $f(y)=0$, then
    \[\sum\limits_{n=1}^\infty f_n(y_n)=0~~\Rightarrow~~\left|\sum\limits_{x_n\neq0}f_n(y_n)\right|\leq\left|\sum\limits_{x_n=0}f_n(y_n)\right|\leq\sum\limits_{x_n=0}\|y_n\|.\]
    The converse follows trivially since we can choose a support functional $f=\{f_n\}_{n\in\mathbb{N}}$ of $x$ such that $f(y)=0$ by taking $f_n=c\psi_n$ for $x_n=0$ where $\psi_n\in J(y_n)$ and $c$ a suitable constant.
    
\end{proof}
We also express the $p\in(1,\infty)$ case in terms of semi-inner products.
\begin{cor}\label{siporthgonality}
    For $1<p<\infty$ and $x=\{x_n\}_{n\in\mathbb{N}},~y=\{y_n\}_{n\in\mathbb{N}}\in\bigoplus\limits_p\mathbb{X}_n$, $x\perp_By$ if and only if there exist semi-inner products $[\cdot,\cdot]_n$ on $\mathbb{X}_n$ such that
    \[\sum\limits_{n=1}^\infty\|x_n\|^{p-2}[x_n,y_n]_n=0.\]
\end{cor}
Let us also note the following fact which will be used later.
\begin{cor}\label{unit rank}
    Let $x=\{x_n\}_{n\in\mathbb{N}}\in\bigoplus_p\mathbb{X}_n$ be such that $x_n=0$ for every $n\neq n_0$. Then for any $y=\{y_n\}_{n\in\mathbb{N}}\in\bigoplus\limits_p\mathbb{X}_n$,
    \begin{enumerate}
        \item If $p\in(1,\infty)$, $x\perp_By$ if and only if $x_{n_0}\perp_By_{n_0}$ and $y\perp_Bx$ if and only if $y_{n_0}\perp_Bx_{n_0}$.
        \item If $p=1$, $x\perp_By$ if and only if \[\inf\{f_{n_0}(y_{n_0}):f_{n_0}\in J(x_{n_0})\}\leq \sum\limits_{n\neq n_0}\|y_n\|.\]
        Also, $y\perp_Bx$ if and only if $y_{n_0}\perp_Bx_{n_0}$.
        \item If $p=0$, $x\perp_By$ if and only if $x_{n_0}\perp_By_{n_0}$ and $y\perp_Bx$ if and only if $\|y_n\|=\|y\|_0$ for some $n\neq n_0$ or $y_{n_0}\perp_Bx_{n_0}$.
    \end{enumerate}
\end{cor}
We now proceed toward characterizing the pointwise symmetry of Birkhoff-James orthogonality in $\bigoplus\limits_{p}\mathbb{X}_n$. We first characterize the left and right symmetric points for $\bigoplus\limits_1\mathbb{X}_n$, $\bigoplus\limits_2\mathbb{X}_n$ and $\bigoplus\limits_0\mathbb{X}_n$ before turning to the case $p\in(1,\infty)\setminus\{2\}$.
\begin{theorem}
    The space $\bigoplus\limits_1\mathbb{X}_n$ has no non-zero left-symmetric point and $x=\{x_n\}_{n\in\mathbb{N}}$ is a right-symmetric point if and only if there exists $n_0\in\mathbb{N}$ such that
    \[x_n=0~~\text{whenever}~~n\neq n_0~~\text{and}~~x_{n_0}\in\mathbb{X}_{n_0}~~\text{is right-symmetric}.\]
\end{theorem}
\begin{proof}
    Let $x$ be a non-zero left symmetric point of $\bigoplus\limits_1\mathbb{X}_n$. If $x_n=0$, find $y\in\mathbb{X}_n$ such that $\|y\|>\|x\|_1$. Then considering $z=\{z_n\}_{n\in\mathbb{N}}\in\bigoplus\limits_{1}\mathbb{X}_n$ given by $z_k=x_k$ if $k\neq n$ and $z_n=y$, we clearly get from Theorem \ref{orthogonality}:
    \[x\perp_Bz,~~z\not\perp_Bx.\]
    Now, if $x_n\neq0$ for every $n$, find $M$ such that
    \begin{align*}
        \sum\limits_{n=1}^M\|x_n\|\neq\sum\limits_{n=M+1}^\infty\|x_n\|.
    \end{align*}
    Consider $z=\{z_n\}_{n\in\mathbb{N}}\in\bigoplus\limits_1\mathbb{X}_n$ given by
    \begin{align*}
        z_n=
        \begin{cases}
            \frac{x_n}{\|x_n\|},~~1\leq n\leq M,\\
            \frac{Mx_n}{2^{n-M}\|x_n\|},~~n>M.
        \end{cases}
    \end{align*}
    Then again by Theorem \ref{orthogonality}, we get $x\perp_Bz$ but $z\not\perp_Bx$.\par
    Again, if $x$ is a right-symmetric point of $\bigoplus\limits_1\mathbb{X}_n$ and $x_1,x_2\neq0$, then without loss of generality, we may assume that $\|x_1\|\leq\|x_2\|$. Now consider $y=\{y_n\}_{n\in\mathbb{N}}\in\bigoplus\limits_1\mathbb{X}_n$ given by $y_1=x_1$ and $y_n=0$ otherwise. Then clearly, $x\not\perp_By$ but $y\perp_Bx$ by Theorem \ref{orthogonality}. Further, if $x$ has only one non-zero component $x_{n_0}$ and $x$ is right-symmetric, then for any $y\in\bigoplus\limits_1\mathbb{X}_n$, with only non-zero component $y_{n_0}$,
    \[x\perp_By~~\Leftrightarrow~~x_n\perp_By_n~~\text{and}~~y\perp_Bx~~\Leftrightarrow~~y_n\perp_Bx_n.\]
    Hence $x_{n_0}$ is right-symmetric. The converse, however, follows trivially from Theorem \ref{orthogonality} and Corollary \ref{unit rank}.
\end{proof}
We now come to the case $p=2$.
\begin{theorem}
    Let $x=\{x_n\}_{n\in\mathbb{N}}\in\bigoplus\limits_2\mathbb{X}_n$. Then
    \begin{enumerate}
        \item $x$ is left-symmetric if and only if $x_n\in\mathbb{X}_n$ is left s.i.p. symmetric (see Definition \ref{sip-commute}) for every $n\in\mathbb{N}$ or there exists $n_0\in\mathbb{N}$ such that $x_n=0$ if $n\neq n_0$ and $x_{n_0}$ is a left-symmetric point of $\mathbb{X}_{n_0}$.
        \item $x$ is right-symmetric if and only if $x_n\in\mathbb{X}_n$ is right s.i.p. symmetric for every $n\in\mathbb{N}$ or there exists $n_0\in\mathbb{N}$ such that $x_n=0$ if $n\neq n_0$ and $x_{n_0}$ is a right-symmetric point of $\mathbb{X}_{n_0}$.
    \end{enumerate}
\end{theorem}
\begin{proof}
    From Corollary \ref{siporthgonality}, we get that for any $y=\{y_n\}_{n\in\mathbb{N}}\in\bigoplus\limits_2\mathbb{X}_n$, $x\perp_By$ if and only if
    \begin{align}\label{l2ortho}
        \sum\limits_{n=1}^\infty[x_n,y_n]_n=0,
    \end{align}
    for some sequence of semi-inner products $[\cdot,\cdot]_n$ on $\mathbb{X}$. The sufficiency for both the parts now follows from \eqref{l2ortho}. To prove the necessity, we assume a contradiction. We prove the result for the left-symmetric case, and the proof for the right-symmetric case follows similarly. Without loss of generality, we therefore assume that there exist $z_1\in\mathbb{X}_1$ and a semi-inner product $[\cdot,\cdot]_1$ such that 
    \[[x_1,z_1]\neq[y_1,z_1]'~~\text{for every semi-inner product}~[\cdot,\cdot].\]
    If $x_m\neq0$ for some $m\neq 1$ then define $y_{\alpha}=\{y_n\}_{n\in\mathbb{N}}\in\bigoplus\limits_2\mathbb{X}_n$ for $\alpha\in\mathbb{C}$ by
    \[y_n=\begin{cases}
         z_1,~~n=1,\\
        \alpha x_m,~~n=m,\\
        0,~~\text{otherwise}.
    \end{cases}\]
    Clearly, $x\perp_B y_\alpha$ if 
    \begin{align*}
        [x_1,z_1]+\alpha\|x_n\|^2=0.
    \end{align*}
    But $y_\alpha\perp_Bx$ if and only if 
    \begin{align*}
        [z_1,x_1]'+\overline{\alpha}\|x_n\|^2=0,
    \end{align*}
    giving the desired contradiction.
\end{proof}
We now proceed to the case $p=0$.
\begin{theorem}
    The space $\bigoplus\limits_0\mathbb{X}_n$ has no non-zero right-symmetric point. A point $x=\{x_n\}_{n\in\mathbb{N}}\in\bigoplus\limits_0\mathbb{X}_n$ is left symmetric if and only if there exists $n_0\in\mathbb{N}$ such that $x_n=0$ whenever $n\neq n_0$ and $x_{n_0}$ is a left-symmetric point of $\mathbb{X}_{n_0}$.
\end{theorem}
\begin{proof}
    For the right-symmetric part, let $x=\{x_n\}_{n\in\mathbb{N}}\in\bigoplus\limits_0\mathbb{X}_n$ be right-symmetric and $\|x_1\|=\|x\|_0\neq 0$ without loss of generality. Now if $m\neq n$, then if $\|x_m\|<\|x\|_0$, we consider $y=\{y_n\}_{n\in\mathbb{N}}\in\bigoplus\limits_0\mathbb{X}_n$ given by
    \begin{align*}
        y_n=\begin{cases}
            -\frac{x_m}{\|x_m\|},~~\text{if}~n=m,\\
            \frac{x_n}{\|x_n\|},~~\text{if}~\|x_n\|=\|x\|_0,\\
            0,~~\text{otherwise}.
        \end{cases}
    \end{align*}
    Theorem \ref{orthogonality} now yields that $y\perp_Bx$ but $x\not\perp_By$. Hence $\|x_n\|=\|x\|_0$ for every $n\in\mathbb{N}$. Since $x\in\bigoplus\limits_0\mathbb{X}_n$, $x=0$.\par
    The sufficiency for the left-symmetric case follows from Corollary \ref{unit rank}. For the necessity, we assume $\|x_1\|=\|x\|_0$ without loss of generality. Then if $x_m\neq0$ for some $m\neq1$, find $y_1\in\mathbb{X}_1$ and $y_m\in\mathbb{X}_m$ such that $x_1\perp_By_1$, $\|y_1\|=1$ and $y_m\not\perp_Bx_m$, $\|y_m\|=2$. We now define $y=\{y_n\}_{n\in\mathbb{N}}\in\bigoplus\limits_0\mathbb{X}_n$ by setting $y_n=0$ for $n\neq 1,m$ yields $x\perp_By$ but $y\perp_Bx$. Hence $x_m=0$ for every $m>1$. That $x_1$ is left-symmetric again follows from Corollary \ref{unit rank}.    
\end{proof}
We finish the section with the $p\in(1,\infty)\setminus\{2\}$ case.
\begin{theorem}
    Let $x=\{x_n\}_{n\in\mathbb{N}}\in\bigoplus\limits_p\mathbb{X}_n$. Then
    \begin{enumerate}
        \item If $x$ is left-symmetric, either there exists $n_0\in\mathbb{N}$ such that $x_n=0$ for $n\neq n_0$ and $x_{n_0}$ is a left-symmetric point of $\mathbb{X}_n$, or there exist $m,k\in\mathbb{N}$ such that $x_n=0$ if $n\neq m,k$, $\|x_m\|=\|x_k\|$, and $x_m,~x_k$ are smooth $p$-left s.i.p. symmetric points of $\mathbb{X}_m$, $\mathbb{X}_k$ respectively.
        \item If $x$ is right-symmetric, either there exists $n_0\in\mathbb{N}$ such that $x_n=0$ for $n\neq n_0$ and $x_{n_0}$ is a right-symmetric point of $\mathbb{X}_n$, or there exist $m,k\in\mathbb{N}$ such that $x_n=0$ if $n\neq m,k$, $\|x_m\|=\|x_k\|$, and $x_m,~x_k$ are $p$-right s.i.p. symmetric points of $\mathbb{X}_m$, $\mathbb{X}_k$ respectively.
    \end{enumerate}
\end{theorem}
\begin{proof}
The sufficiency in both the cases follows directly from computation. We establish the necessity for the left-symmetric case and the proof for the right-symmetric case follows in a similar way. \par
Observe that by Corollary \ref{unit rank}, if $x$ is left-symmetric, $x_n\in\mathbb{X}_n$ is left-symmetric for every $n\in\mathbb{N}$. Further, if there are more than one non-zero components of $x$, without loss of generality, we can assume that $x_1,~x_2\neq0$. If $\|x_1\|>\|x_2\|$, then for $\alpha>0$, define $y_{\alpha}:=\{y_n\}_{n\in\mathbb{N}}\in\bigoplus\limits_p\mathbb{X}_n$ given by
    \[y_n=\begin{cases}
        x_1,~~n=1,\\
        -\alpha x_2,~~n=2,\\
        0,~~\text{otherwise}.
    \end{cases}\]
    Hence, by Corollary \ref{siporthgonality}, $x\perp_By_\alpha$ if and only if $\|x_1\|^p=\alpha\|x_2\|^p$ and $y_\alpha\perp_Bx$ if and only if $\|x_1\|^p=\alpha^{p-1}\|x_2\|^p$, giving the desired contradiction as $p\neq2$. Further if additionally $x_3\neq0$, we can assume that $\|x_1\|=\|x_2\|=\|x_3\|$. Now for $\alpha,~\beta\in\mathbb{K}$, set $z_{\alpha,\beta}=\{z_n\}_{n\in\mathbb{N}}\in\bigoplus\limits_p\mathbb{X}_n$ given by
    \[z_n=
    \begin{cases}
        x_1,~~n=1,\\
        \alpha x_2,~~n=2,\\
        \beta x_3,~~n=3,\\
        0,~~\text{otherwise}.
    \end{cases}\]
    Thus by Corollary \ref{siporthgonality}, $x\perp_B z_{\alpha,\beta}$ if and only if
    \begin{align*}
        \|x_1\|^2+\alpha\|x_2\|^2+\beta\|x_3\|^2=0~~\Rightarrow~~1+\alpha+\beta=0.
    \end{align*}
    Also, $z_{\alpha,\beta}\perp_Bx$ if and only if 
    \begin{align*}
        1+|\alpha|^{p-1}\overline{\sgn(\alpha)}+|\beta|^{p-1}\overline{\sgn(\beta)}=0.
    \end{align*}
    Setting $\alpha=\beta=-\frac{1}{2}$, yields that $x$ cannot be left-symmetric. Thus, without loss of generality, we assume $\|x_1\|=\|x_2\|\neq0$ and $x_n=0$ for $n>2$. Let $\mathcal{O}_1:=\{y_1\in\mathbb{X}_1:x_1\not\perp_By_1\}$. Then $\mathcal{O}_1$ is an open subset of $\mathbb{X}_1$. Since the collection of smooth points is dense in $\mathbb{X}_1$, find $y_1\in\mathbb{O}_1$ smooth. Then for any semi-inner product $[\cdot,\cdot]$ on $\mathbb{X}_1$,
    \[[x_1,y_1]+\alpha\|x_2\|^2=0~~\Rightarrow~~x\perp_By,\]
    where $y_1$ is as chosen before, $y_2=\alpha x_2$ and $y_n=0$ otherwise. But then $y\perp_Bx$ giving:
    \[\|y_1\|^{p-2}[y_1,x_1]'+|\alpha|^p\frac{\|x_2\|^p}{\alpha}=0,\]
    for some semi-inner product $[\cdot,\cdot]'$ on $\mathbb{X}_1$. Hence, 
    \begin{align*}
        \|y_1\|^{p-2}[y_1,x_1]'=|[x_1,y_1]|^p\frac{\|x_2\|^{2-p}}{[x_1,y_1]}.
    \end{align*}
    Since $y_1$ is smooth, $[y_1,x_1]'$ is unique for any semi-inner product $[\cdot,\cdot]'$ on $\mathbb{X}_1$. Hence $[x_1,y_1]$ is uniquely determined on the region $\{y_1:y_1\in\mathcal{O}_1,~y_1~\text{smooth}\}$ for any semi-inner product $[\cdot,\cdot]$ on $\mathbb{X}_1$. Hence $[x_1,y_1]$ is uniquely determined for $y_1\in\mathcal{O}_1$. Since $\Span(\mathcal{O}_1)=\mathbb{X}_1$, $[x_1,y_1]$ is uniquely determined on $\mathbb{X}_1$, showing that $x_1$ is smooth. Further, given $y_1\in\mathbb{X}_1$, there exists a semi-inner product $[\cdot,\cdot]'$ on $\mathbb{X}_1$ such that
    \begin{align*}
        (\|x_1\|\|y_1\|)^{p-2}[x_1,y_1][y_1,x_1]'=|[x_1,y_1]|^p~\Rightarrow~[y,x]'=\left|\frac{[x_1,y_1]}{\|x_1\|\|y_1\|}\right|^{p-2}\overline{[x_1,y_1]},
    \end{align*}
finishing the proof for the left-symmetric case.  The necessity for the right-symmetric case follows similarly.   
\end{proof}


\begin{thebibliography}{100}



\bibitem{B} G. Birkhoff,  \textit{``Orthogonality in linear metric spaces"}, \texttt{Duke Math. J., 1 (1935), 169-172.}

\bibitem{me}
B. Bose, \textit{``Birkhoff-James orthogonality and its point-wise symmetry in some function spaces"}, \texttt{arXiv:2205.13078 [math.FA]}

\bibitem{usseq}
B. Bose, S. Roy, D. Sain, \textit{``Birkhoff-James orthogonality and its local symmetry in some sequence spaces"}, \texttt{Revista de la Academia de Ciencias Exactas, F\'isicas y Naturales. Series A. Mathem\'atical, 117(3), No. 93, 2023.}




\bibitem{CSS} A. Chattopadhyay, D. Sain, T. Senapati, \textit{ ``Characterization of symmetric points in $ l_p^n $-spaces"},\texttt{ Lin. Mult. Alg., 69 (2021), No. 16, 2998-3009}.

\bibitem{approx}
J. Chmieli\'nski, D. Khurana, D. Sain, \textit{``Approximate smoothness in normed linear spaces"}, \texttt{Banach J. Math. Anal., 17 (2023), No. 3, 41.}

\bibitem{dkp}
P. Ghosh, D. Sain and K. Paul, \textit{``On symmetry of Birkhoff-James orthogonality of linear operators"}, \texttt{Adv. Oper. Theory, 2 (2017), 428-434.}

\bibitem{1}
 P. Ghosh, K. Paul and D. Sain, \textit{``Symmetric properties of orthogonality of
linear operators on $(\mathbb{R}^
n,\|.\|_1)$"}, \texttt{Novi Sad J. Math., 47 (2017), 41-46.}




\bibitem{james2} R.C. James, \textit{``Inner product in normed linear spaces"}, \texttt{Bull. Amer. Math. Soc., 53 (1947), 559-566.}

\bibitem{james} R.C. James, \textit{``Orthogonality and linear functionals in normed linear spaces"}, \texttt{Trans. Amer. Math. Soc., 61 (1947), 265-292.}

\bibitem{3}
N. Komuro, K.-S. Saito and R. Tanaka, \textit{``Left symmetric points for Birkhoff
orthogonality in the preduals of von Neumann algebras"}, \texttt{Bull. Aust. Math.
Soc., 98 (2018), 494-501.}

\bibitem{4}
 N. Komuro, K.-S. Saito and R. Tanaka, \textit{``Symmetric points for (strong)
Birkhoff orthogonality in von Neumann algebras with applications to preserver problems"}, \texttt{J. Math. Anal. Appl., 463 (2018), 1109-1131.}

\bibitem{5}
 N. Komuro, K.-S. Saito and R. Tanaka, \textit{``On symmetry of Birkhoff orthogonality in the positive cones of $C^*$-algebras with applications"}, \texttt{J. Math. Anal.
Appl., 474 (2019), 1488–1497.}

\bibitem{KP} K. Paul, A. Mal and P. W\'{o}jcik, \textit{``Symmetry of Birkhoff-James orthogonality of operators defined between infinite dimensional Banach spaces"}, \texttt{Linear Algebra Appl. {\bf}563 (2019), 142-153.}

\bibitem{Sain2} D. Sain, \textit{``Birkhoff-James orthogonality of linear operators on finite dimensional Banach spaces"}, \texttt{J. Math. Anal. Appl., 447 (2017), No. 2,  860-866.}


\bibitem{Sain} D. Sain, \textit{``On the norm attainment set of a bounded linear operator"}, \texttt{J. Math. Anal. Appl., 457 (2018), No. 1, 67-76.}



\bibitem{8}
D. Sain, P. Ghosh and K. Paul, \textit{``On symmetry of Birkhoff-James orthogonality of linear operators on finite-dimensional real Banach spaces"}, \texttt{Oper.
Matrices, 11 (2017), 1087-1095.}

\bibitem{10}
D. Sain, K. Paul, A. Mal, A. Ray, \textit{``A complete characterization of smoothness
in the space of bounded linear operators"}, \texttt{Linear Multilinear Algebra, 2019,
doi.org/10.1080/03081087.2019.1586824.}

\bibitem{SRBB} D. Sain, S. Roy, S. Bagchi and V. Balestro, \textit{``A study of symmetric points in Banach spaces"}, \texttt{Linear Multilinear Algebra, (2020), https://doi.org/10.1080/03081087.2020.1749541.}




\bibitem{turnsek}
 A. Turnsek, \textit{``On operators preserving James’ orthogonality"}, \texttt{Linear Algebra and its Applications, 407 (2005), 189-195.}

\bibitem{12}
A. Turn\^sek, \textit{``A remark on orthogonality and symmetry of operators in B(H)"},
\texttt{Linear Algebra Appl., 535 (2017), 141-150.}

 
\end{thebibliography}
\end{document}